\pdfoutput=1
\RequirePackage{ifpdf}
\ifpdf % We~are running pdfTeX in pdf mode
\documentclass[pdftex]{sigma}%,draft
\else
\documentclass{sigma}
\fi

\usepackage{array,booktabs,multirow}
\usepackage{tikz}
\usetikzlibrary{matrix,arrows}

\def\A{\mathbf A}
\def\E{\mathbf E}
\def\C{\mathbb C}
\def\End{{\rm End}}
\def\H{\mathcal H}
\def\I{\mathbf I}
\def\K{\mathfrak K_\omega}
\def\Mat{{\rm Mat}}
\def\M{\mathcal M}

\def\T{\mathcal T}

\def\U{{\rm U}(\mathfrak{sl}_2)}
\def\e{\varepsilon}

\numberwithin{equation}{section}

\newtheorem{Theorem}{Theorem}[section]
\newtheorem*{Theorem*}{Theorem}
\newtheorem{Corollary}[Theorem]{Corollary}
\newtheorem{Lemma}[Theorem]{Lemma}
\newtheorem{Proposition}[Theorem]{Proposition}
 { \theoremstyle{definition}
\newtheorem{Definition}[Theorem]{Definition}

 }

\begin{document}
\allowdisplaybreaks

\newcommand{\arXivNumber}{2106.06857}

\renewcommand{\PaperNumber}{017}

\FirstPageHeading

\ShortArticleName{The Clebsch--Gordan Rule for $U(\mathfrak{sl}_2)$, the Krawtchouk Algebras and the Hamming Graphs}

\ArticleName{The Clebsch--Gordan Rule for $\boldsymbol{U(\mathfrak{sl}_2)}$,\\ the Krawtchouk Algebras and the Hamming Graphs}

\Author{Hau-Wen HUANG}

\AuthorNameForHeading{H.-W.~Huang}

\Address{Department of Mathematics, National Central University, Chung-Li 32001, Taiwan}
\Email{\href{mailto:hauwenh@math.ncu.edu.tw}{hauwenh@math.ncu.edu.tw}}

\ArticleDates{Received October 03, 2022, in final form March 22, 2023; Published online April 04, 2023}

\Abstract{Let $D\geq 1$ and $q\geq 3$ be two integers. Let $H(D)=H(D,q)$ denote the $D$-dimensional Hamming graph over a $q$-element set. Let $\T(D)$ denote the Terwilliger algebra of $H(D)$. Let $V(D)$ denote the standard $\T(D)$-module. Let $\omega$ denote a complex scalar. We~consider a unital associative algebra~$\K$ defined by generators and relations. The generators are $A$ and $B$. The relations are $A^2 B-2 ABA +B A^2 =B+\omega A$, $B^2A-2 BAB+AB^2=A+\omega B$. The algebra~$\K$ is the case of the Askey--Wilson algebras corresponding to the Krawtchouk polynomials. The algebra~$\K$ is isomorphic to $\U$ when $\omega^2\not=1$. We view~$V(D)$ as a~\smash{$\mathfrak{K}_{1-\frac{2}{q}}$}-module. We apply the Clebsch--Gordan rule for $\U$ to decompose~$V(D)$ into a~direct sum of irreducible $\T(D)$-modules.}

\Keywords{Clebsch--Gordan rule; Hamming graph; Krawtchouk algebra; Terwilliger algebra}

\Classification{05E30; 16G30; 16S30; 33D45}

\section{Introduction}\label{s:introduction}

Throughout this paper, we adopt the following conventions: Fix an integer $q\geq 3$. Let $\C$ denote the complex number field. An algebra is meant to be a unital associative algebra. An algebra homomorphism is meant to be a unital algebra homomorphism. A subalgebra has the same unit as the parent algebra.
In an algebra the commutator $[x,y]$ of two elements $x$ and $y$ is defined as
$[x,y]=xy-yx$. Note that every algebra has a Lie algebra structure with Lie bracket given by the commutator.

Recall that $\mathfrak{sl}_2(\C)$ is a three-dimensional Lie algebra over $\C$ with a basis $e$, $f$, $h$ satisfying
\begin{gather*}
[h,e]=2e,\qquad
[h,f]=-2f,\qquad
[e,f]=h.
\end{gather*}

\begin{Definition}\label{defn:U}
The {\it universal enveloping algebra $\U$ of $\mathfrak{sl}_2$} is an algebra over $\C$ generated by $E$, $F$, $H$ subject to the relations
\begin{gather*}
[H,E]=2E,\qquad
[H,F]=-2F,\qquad
[E,F]=H.
\end{gather*}
\end{Definition}

Using Definition \ref{defn:U}, it is straightforward to verify the following lemma:

\begin{Lemma}\label{lem:Vn_U}
Given any integer $n\geq 0$ there exists an $(n+1)$-dimensional $\U$-module $L_n$
that has a basis $\{v_i\}_{i=0}^n$ such that
\begin{alignat*}{4}
&E v_i=(n-i+1) v_{i-1}\qquad&&
\text{for}\ i=1,2,\dots,n,\qquad&&
E v_0=0,&
\\
&F v_i=(i+1) v_{i+1}\qquad&&
\text{for}\ i=0,1,\dots,n-1,\qquad&&
F v_n=0,&
\\
&H v_i=(n-2i) v_i\qquad &&\text{for}\ i=0,1,\dots,n.&&&
\end{alignat*}
\end{Lemma}

Note that the $\U$-module $L_n$ is irreducible for any integer $n\geq 0$. Furthermore, the finite-dimensional irreducible $\U$-modules are classified as follows:

\begin{Lemma}\label{lem:classification_Umodule}
For any integer $n\geq 0$, each $(n+1)$-dimensional irreducible $\U$-module is isomorphic to $L_n$.
\end{Lemma}

\begin{proof}
See \cite[Section V.4]{kassel} for example.
\end{proof}

It is well known that the universal enveloping algebra of a Lie algebra is a Hopf algebra. For example, see \cite[Section 5]{Hopf65}.

\begin{Lemma}\label{lem:Hopf_U}
The algebra $\U$ is a Hopf algebra on which the counit $\e\colon \U\to \C$, the antipode $S\colon \U\to \U$ and the comultiplication $\Delta\colon \U\to \U\otimes \U$ are given by
\begin{alignat*}{4}
&\e(E)=0,\qquad&&
\e(F) =0,\qquad&&
\e(H)=0,&
\\
&S(E)=-E,\qquad&&
S(F) =-F,\qquad&&
S(H)=-H,&
\\
&\Delta(E)=E\otimes 1+1\otimes E,
\qquad&&
\Delta(F)=F\otimes 1+1\otimes F,
\qquad&&
\Delta(H)=H\otimes 1+1\otimes H.&
\end{alignat*}
\end{Lemma}

Every $\U\otimes \U$-module can be viewed as a $\U$-module via the comultiplication of~$\U$.
The Clebsch--Gordan rule for $\U$ is as follows:

\begin{Theorem}\label{thm:CGrule_U}
For any integers $m,n\geq 0$, the $\U$-module $L_m\otimes L_n$ is isomorphic to
\[
\bigoplus_{p=0}^{\min\{m,n\}}L_{m+n-2p}.
\]
\end{Theorem}
\begin{proof}
See \cite[Section V.5]{kassel} for example.
\end{proof}

For the rest of this paper, let $\omega$ denote a scalar taken from $\C$.

\begin{Definition}\label{defn:K}
The {\it Krawtchouk algebra} $\K$ is an algebra over $\C$ generated by $A$ and $B$ subject to the relations
\begin{gather}
A^2 B-2 ABA +B A^2
=
B+\omega A,
\label{R1}
\\
B^2A-2 BAB+AB^2
=
A+\omega B.
\label{R2}
\end{gather}
\end{Definition}
The algebra~$\K$ is the case of the Askey--Wilson algebra corresponding to the Krawtchouk polynomials \cite[Lemma 7.2]{Vidunas:2007}.
Define $C$ to be the following element of $\K$:
\begin{gather*}
C=[A,B].
\end{gather*}

\begin{Lemma}\label{lem:presentation_K}
The algebra~$\K$ has a presentation with the generators $A$, $B$, $C$ and the relations
\begin{align}
[A,B] &=C,
\label{R1_v2}
\\
[A,C]&=B+\omega A,
\label{R2_v2}
\\
[C,B]&=A+\omega B.
\label{R3_v2}
\end{align}
\end{Lemma}
\begin{proof}
The relation \eqref{R1_v2} is immediate from the setting of $C$. Using \eqref{R1_v2}, the relations \eqref{R1} and \eqref{R2} can be written as \eqref{R2_v2} and \eqref{R3_v2}, respectively. The lemma follows.
\end{proof}

Let $\mathcal K_\omega$ denote a three-dimensional Lie algebra over $\C$ with a basis $a$, $b$, $c$ satisfying
\begin{gather*}
[a,b]=c,
\qquad
[a,c]=b+\omega a,
\qquad
[c,b]=a+\omega b.
\end{gather*}
By Lemma~\ref{lem:presentation_K}, the algebra~$\K$ is the universal enveloping algebra of $\mathcal K_\omega$. There is a connection between $\K$ and $\U$:

\begin{Theorem}\label{thm:KtoU}
There exists a unique algebra homomorphism $\zeta\colon \K\to \U$ that sends
\begin{gather*}
A\mapsto \frac{1+\omega}{2} E+\frac{1-\omega}{2} F-\frac{\omega}{2}H,
%\label{zeta:A}
\qquad
B\mapsto \frac{1}{2} H,
%\label{zeta:B}
\qquad
C\mapsto -\frac{1+\omega}{2} E+\frac{1-\omega}{2} F.
%\label{zeta:C}
\end{gather*}
Moreover, if $\omega^2\not=1$, then $\zeta$ is an isomorphism and its inverse sends
\begin{gather*}
E\mapsto \frac{1}{1+\omega}A+\frac{\omega}{1+\omega} B-\frac{1}{1+\omega}C,
%\label{zetain:E}
\qquad
F\mapsto \frac{1}{1-\omega}A+\frac{\omega}{1-\omega} B+\frac{1}{1-\omega}C,
%\label{zetain:F}
\qquad
H\mapsto 2 B.
%\label{zetain:H}
\end{gather*}
\end{Theorem}

\begin{proof}
It is routine to verify the result by using Definition \ref{defn:U} and Lemma~\ref{lem:presentation_K}. Here we provide another proof by applying \cite[Lemmas~2.12 and~2.13]{K&sl2}.

Let $\sigma\colon \mathfrak{sl}_2(\C)\to \U$ denote the canonical Lie algebra homomorphism that sends $e$, $f$, $h$ to $E$, $F$, $H$, respectively.
 Let $\tau\colon \mathcal K_\omega\to \K$ denote the canonical Lie algebra homomorphism that sends $a$, $b$, $c$ to $A$, $B$, $C$, respectively.
By \cite[Lemma 2.12]{K&sl2}, there exists a unique Lie algebra homomorphism $\phi\colon \mathcal K_\omega \to \mathfrak{sl}_2(\C)$ that sends
\begin{gather*}
a \mapsto \frac{1+\omega}{2} e+\frac{1-\omega}{2} f-\frac{\omega}{2} h,
\qquad
b \mapsto \frac{1}{2} h,
\qquad
c \mapsto -\frac{1+\omega}{2} e+\frac{1-\omega}{2} f.
\end{gather*}
Applying the universal property of $\K$ to the Lie algebra homomorphism $
\sigma\circ \phi$, this gives the algebra homomorphism $\zeta$. Suppose that $\omega^2\not=1$. Then $\phi\colon \mathcal K_\omega \to \mathfrak{sl}_2(\C)$ is a Lie algebra isomorphism by \cite[Lemma 2.13]{K&sl2}.
Applying the universal property of $\U$ to the Lie algebra homomorphism $\tau\circ \phi^{-1}$, this gives the inverse of $\zeta$.
\end{proof}

In this paper, we relate the above algebraic results to the Hamming graphs. We now recall the definition of Hamming graphs.
Let $X$ denote a $q$-element set and let $D$ be a positive integer.
The {\it $D$-dimensional Hamming graph $H(D)=H(D,q)$ over $X$} is a simple graph whose vertex set is $X^D$ and $x,y\in X^D$ are adjacent if and only if $x$, $y$ differ in exactly one coordinate. Let~$\partial$ denote the path-length distance function for $H(D)$.
Let $\Mat_{X^D}(\C)$ stand for the algebra consisting of the square matrices over $\C$ indexed by $X^D$.

The adjacency matrix $\A(D)\in \Mat_{X^D}(\C)$ of $H(D)$ is the $0$-$1$ matrix such that
\[
\A(D)_{xy}=1\qquad
\text{if and only if}\qquad
\partial(x,y)=1
\]
for all $x,y\in X^D$.
Fix a vertex $x\in X^D$. The dual adjacency matrix $\A^*(D)\in \Mat_{X^D}(\C)$ of $H(D)$ with respect to $x$ is a diagonal matrix given by
\[
\A^*(D)_{yy}=D(q-1)-q\cdot \partial(x,y)
\]
for all $y\in X^D$.
The Terwilliger algebra $\T(D)$ of $H(D)$ with respect to $x$ is the subalgebra of $\Mat_{X^D}(\C)$ generated by $\A(D)$ and $\A^*(D)$ \cite{TerAlgebraI,TerAlgebraII,TerAlgebraIII}.
Let $V(D)$ denote the vector space consisting of all column vectors over $\C$ indexed by $X^D$. The vector space $V(D)$ has a natural $\T(D)$-module structure and it is called the standard $\T(D)$-module.

In \cite{TerAlgebraIII}, Terwilliger employed the endpoints, dual endpoints, diameters and auxiliary parameters to describe the irreducible modules for
the known families of thin $Q$-polynomial distance-regular graphs with unbounded diameter.
In \cite{Doob1997}, Tanabe gave a recursive construction of irreducible modules for the Doob graphs and his method can be adjusted to the case of $H(D)$. In \cite{hypercube2002}, Go gave a decomposition of the standard module for the hypercube. In \cite{semidefinite2006}, Gijswijt, Schrijver and Tanaka described a decomposition of $V(D)$ in terms of the block-diagonalization of $\T(D)$. In \cite{Hamming2006}, Levstein, Maldonado and Penazzi applied the representation theory of ${\rm GL}_2(\C)$ to determine the structure of $\T(D)$. In \cite{Manila}, it was shown that $V(D)$ can be viewed as a $\mathfrak{gl}_2(\C)$-module as well as a $\mathfrak{sl}_2(\C)$-module. In \cite{Hamming:2021}, Bernard, Cramp\'{e}, and Vinet found a decomposition of $V(D)$ by generalizing the result on the hypercube.

In this paper, we view $V(D)$ as a $\mathfrak{K}_{1-\frac{2}{q}}$-module as well as a $\U$-module in light of Theorem~\ref{thm:KtoU}. Subsequently, we apply Theorem~\ref{thm:CGrule_U} to prove the following results:

\begin{Proposition}\label{prop:Kpk}
Let $D$ be a positive integer.
For any integers $p$ and $k$ with $0\leq p\leq D$ and $0\leq k\leq \big\lfloor \frac{p}{2}\big\rfloor$, there exists a $(p-2k+1)$-dimensional irreducible $\T(D)$-module $L_{p,k}(D)$ satisfying the following conditions:
\begin{enumerate}\itemsep=0pt
\item[$(i)$] There exists a basis for $L_{p,k}(D)$ with respect to which the matrices representing $\A(D)$ and~$\A^*(D)$ are
\begin{gather*}
\begin{pmatrix}
\alpha_0 &\gamma_1 & & &{\bf 0}
\\
\beta_0 &\alpha_1 &\gamma_2
\\
&\beta_1 &\alpha_2 &\ddots
 \\
& &\ddots &\ddots &\gamma_{p-2k}
 \\
{\bf 0} & & &\beta_{p-2k-1} &\alpha_{p-2k}
\end{pmatrix}\!,
\qquad
\begin{pmatrix}
\theta_0 & & & &{\bf 0}
\\
 &\theta_1
\\
 & &\theta_2
 \\
 & & &\ddots
 \\
{\bf 0} & & & &\theta_{p-2k}
\end{pmatrix}\!,
\end{gather*}
respectively.

\item[$(ii)$] There exists a basis for $L_{p,k}(D)$ with respect to which the matrices representing $\A(D)$ and~$\A^*(D)$ are
\begin{gather*}
\begin{pmatrix}
\theta_0 & & & &{\bf 0}
\\
 &\theta_1
\\
 & &\theta_2
 \\
 & & &\ddots
 \\
{\bf 0} & & & &\theta_{p-2k}
\end{pmatrix}\!,
\qquad
\begin{pmatrix}
\alpha_0 &\gamma_1 & & &{\bf 0}
\\
\beta_0 &\alpha_1 &\gamma_2
\\
&\beta_1 &\alpha_2 &\ddots
 \\
& &\ddots &\ddots &\gamma_{p-2k}
 \\
{\bf 0} & & &\beta_{p-2k-1} &\alpha_{p-2k}
\end{pmatrix}\!,
\end{gather*}
respectively.
\end{enumerate}
Here the parameters $\{\alpha_i\}_{i=0}^{p-2k}$,
$\{\beta_i\}_{i=0}^{p-2k-1}$, $\{\gamma_i\}_{i=1}^{p-2k}$, $\{\theta_i\}_{i=0}^{p-2k}$ are as follows:
\begin{alignat*}{3}
&\alpha_i=(q-2)(i+k)+p-D\qquad
&&\text{for}\ i=0,1,\dots,p-2k,&
\\
&\beta_i=i+1\qquad
&&\text{for}\ i=0,1,\dots,p-2k-1,&
\\
&\gamma_i=(q-1)(p-i-2k+1)\qquad
&&\text{for}\ i=1,2,\dots,p-2k,&
\\
&\theta_i=q(p-i-k)-D\qquad
&&\text{for} \ i=0,1,\dots,p-2k.&
\end{alignat*}
\end{Proposition}

Given a vector space $W$ and a positive integer $p$, we let
\[
p\cdot W=\underbrace{W\oplus W\oplus \cdots \oplus W}_
{\text{{\scriptsize $p$ copies of $W$}}}.
\]

\begin{Theorem}\label{thm:dec_T(D)module}
Let $D$ be a positive integer. Then the standard $\T(D)$-module $V(D)$ is isomorphic to
\begin{gather*}
\bigoplus_{p=0}^D\bigoplus_{k=0}^{\lfloor \frac{p}{2}\rfloor}
\frac{p-2k+1}{p-k+1}{D\choose p}{p\choose k}(q-2)^{D-p}\cdot L_{p,k}(D).
\end{gather*}
\end{Theorem}

The algebra $\T(D)$ is a finite-dimensional semisimple algebra. Following from \cite[Theorem~25.10]{Wedderburn1962}, Theorem~\ref{thm:dec_T(D)module} implies the following classification of irreducible $\T(D)$-modules:

\begin{Theorem}\label{thm:T(D)irrmodule}
Let $D$ be a positive integer. Let $\mathbf P(D)$ denote the set consisting of all pairs $(p,k)$ of integers with $0\leq p\leq D$ and $0\leq k\leq \big\lfloor \frac{p}{2}\big\rfloor$. Let $\mathbf M(D)$ denote the set of all isomorphism classes of irreducible $\T(D)$-modules. Then there exists a bijection $\mathcal E\colon \mathbf P(D)\to \mathbf M(D)$ given by
\begin{gather*}
(p,k) \mapsto \text{the isomorphism class of $L_{p,k}(D)$}
\end{gather*}
for all $(p,k)\in \mathbf P(D)$.
\end{Theorem}

The paper is organized as follows:
In Section~\ref{s:Hopf}, we give the preliminaries on the algebra~$\K$.
In Section~\ref{s:Hamming}, we prove Proposition~\ref{prop:Kpk} and Theorems~\ref{thm:dec_T(D)module}, \ref{thm:T(D)irrmodule} by using Theorem~\ref{thm:CGrule_U}.
In~Appendix~\ref{s:restate}, we give the equivalent statements of Proposition~\ref{prop:Kpk} and Theorems~\ref{thm:dec_T(D)module}, \ref{thm:T(D)irrmodule}.

\section{The Krawtchouk algebra}\label{s:Hopf}

\subsection[Finite-dimensional irreducible K\_w-modules]{Finite-dimensional irreducible $\boldsymbol{\K}$-modules}

Recall the $\U$-module $L_n$ from Lemma~\ref{lem:Vn_U}.
Recall the algebra homomorphism $\zeta\colon \K\to \U$ form Theorem~\ref{thm:KtoU}.
Each $\U$-module can be viewed as a $\K$-module by pulling back via $\zeta$.
We express the $\U$-module $L_n$ as a $\K$-module as follows:

\begin{Lemma}\label{lem:Vn_K}
For any integer $n\geq 0$, the matrices representing $A$, $B$, $C$ with respect to the basis $\{v_i\}_{i=0}^n$ for the $\K$-module $L_n$ are
\begin{gather*}
\begin{pmatrix}
\alpha_0 &\gamma_1 & & &{\bf 0}
\\
\beta_0 &\alpha_1 &\gamma_2
\\
&\beta_1 &\alpha_2 &\ddots
 \\
& &\ddots &\ddots &\gamma_n
 \\
{\bf 0} & & &\beta_{n-1} &\alpha_n
\end{pmatrix}\!,
\quad
\begin{pmatrix}
\theta_0 & & & &{\bf 0}
\\
 &\theta_1
\\
 & &\theta_2
 \\
 & & &\ddots
 \\
{\bf 0} & & & &\theta_n
\end{pmatrix}\!,
\quad
\begin{pmatrix}
0 &-\gamma_1 & & &{\bf 0}
\\
\beta_0 &0 &-\gamma_2
\\
&\beta_1 &0 &\ddots
 \\
& &\ddots &\ddots &-\gamma_n
 \\
{\bf 0} & & &\beta_{n-1} &0
\end{pmatrix}
\end{gather*}
respectively, where
\begin{alignat*}{3}
&\alpha_i=\frac{(2i-n)\omega}{2}\qquad
&&\text{for}\ i=0,1,\dots,n,&
\\
&\beta_i=\frac{(i+1)(1-\omega)}{2}\qquad
&&\text{for}\ i=0,1,\dots,n-1,&
\\
&\gamma_i=\frac{(n-i+1)(1+\omega)}{2}\qquad
&&\text{for}\ i=1,2,\dots,n,&
\\
&\theta_i=\frac{n}{2}-i\qquad
&&\text{for}\ i=0,1,\dots,n.&
\end{alignat*}
\end{Lemma}

The finite-dimensional irreducible $\K$-modules are classified as follows:

{\samepage\begin{Theorem}\label{thm:classification}\quad
\begin{enumerate}\itemsep=-1pt
\item[$(i)$] If $\omega=-1$, then
any finite-dimensional irreducible $\K$-module $V$ is of dimension one and there is a scalar $\mu\in \C$ such that
$Av=\mu v$, $Bv=\mu v$
 for all $v\in V$.

\item[$(ii)$] If $\omega=1$, then any finite-dimensional irreducible $\K$-module $V$ is of dimension one and there is a scalar $\mu\in \C$ such that
$Av=\mu v$, $Bv=-\mu v$ for all $v\in V$.

\item[$(iii)$] If $\omega^2\not=1$, then $L_n$ is the unique $(n+1)$-dimensional irreducible $\K$-module up to isomorphism for every integer $n\geq 0$.
\end{enumerate}
\end{Theorem}}

\begin{proof}
(i) Let $n\geq 0$ be an integer. Let $V$ denote an $(n+1)$-dimensional irreducible $\mathfrak{K}_{-1}$-module. Since the trace of the left-hand side of \eqref{R1} on $V$ is zero, the elements $A$ and $B$ have the same trace on $V$. If $n=0$ then there exists a scalar $\mu\in \C$ such that $Av=Bv=\mu v$ for all $v\in V$.

To see Theorem~\ref{thm:classification}(i), it remains to assume that $n\geq 1$ and we seek a contradiction. Applying the method proposed in \cite{Huang:2015,Huang:BImodule,SH:2019-1}, there exists a basis $\{u_i\}_{i=0}^n$ for $V$ with respect to which the matrices representing $A$ and $B$ are of the forms
\begin{gather*}\setlength{\tabcolsep}{2.0pt}%\renewcommand{\arraystretch}{1.10}
\begin{pmatrix}
\theta_0 & & & &{\bf 0}
\\
1 &\theta_1
\\
&1 &\theta_2
 \\
& &\ddots &\ddots
 \\
{\bf 0} & & &1 &\theta_n
\end{pmatrix},
\qquad
\begin{pmatrix}
\theta_0 &\varphi_1 & & &{\bf 0}
\\
 &\theta_1 &\varphi_2
\\
 & &\theta_2 &\ddots
 \\
 & & &\ddots &\varphi_n
 \\
{\bf 0} & & & &\theta_n
\end{pmatrix},
\end{gather*}
respectively. Here $\{\theta_i\}_{i=0}^n$ is an arithmetic sequence with common difference $-1$ and the sequence $\{\varphi_i\}_{i=1}^n$ satisfies
$\varphi_{i-1}-2\varphi_i+\varphi_{i+1}=0$, $1\leq i\leq n$,
where $\varphi_0$ and $\varphi_{n+1}$ are interpreted as zero. Solving the above recurrence yields that $\varphi_i=0$ for all $i=1,2,\dots,n$. Thus the subspace of $V$ spanned by $\{u_i\}_{i=1}^n$ is a nonzero $\mathfrak{K}_{-1}$-module, which is a contradiction to the irreducibility of $V$.

(ii) Using Definition \ref{defn:K}, it is routine to verify that there exists a unique algebra isomorphism $\mathfrak K_{-1}\to \mathfrak K_1$ that sends $A$ to $A$ and $B$ to $-B$.
Theorem~\ref{thm:classification}(ii) follows from Theorem~\ref{thm:classification}(i) and the above isomorphism.

(iii) Theorem~\ref{thm:classification}(iii) follows immediate from Lemma~\ref{lem:classification_Umodule} and Theorem~\ref{thm:KtoU}.
\end{proof}

\begin{Lemma}\label{lem:auto_K}
There exists a unique algebra automorphism of $\K$ that sends
$A \mapsto B$, $B \mapsto A$, $C \mapsto -C$.
\end{Lemma}

\begin{proof}
It is routine to verify the lemma by using Lemma~\ref{lem:presentation_K}.
\end{proof}

\begin{Lemma}\label{lem2:Vn_K}
Suppose that $\omega^2\not=1$. For any integer $n\geq 0$,
there exists a basis for the $\K$-mod\-ule~$L_n$ with respect to which
the matrices representing $A$, $B$, $C$ are
\begin{gather*}%\renewcommand{\arraystretch}{1.1}
\setlength{\arraycolsep}{4.5pt}
\begin{pmatrix}
\theta_0 & & & &{\bf 0}
\\
 &\theta_1
\\
 & &\theta_2
 \\
 & & &\ddots
 \\
{\bf 0} & & & &\theta_n
\end{pmatrix}\!,
\quad\
\begin{pmatrix}
\alpha_0 &\gamma_1 & & &{\bf 0}
\\
\beta_0 &\alpha_1 &\gamma_2
\\
&\beta_1 &\alpha_2 &\ddots
 \\
& &\ddots &\ddots &\gamma_n
 \\
{\bf 0} & & &\beta_{n-1} &\alpha_n
\end{pmatrix}\!,
\quad\
\begin{pmatrix}
0 &\gamma_1 & & &{\bf 0}
\\
-\beta_0 &0 &\gamma_2
\\
&-\beta_1 &0 &\ddots
 \\
& &\ddots &\ddots &\gamma_n
 \\
{\bf 0} & & &-\beta_{n-1} &0
\end{pmatrix}
\end{gather*}
respectively, where
\begin{alignat*}{3}
&\alpha_i=\frac{(2i-n)\omega}{2}\qquad
&&\text{for} \ i=0,1,\dots,n,&
\\
&\beta_i=\frac{(i+1)(1-\omega)}{2}\qquad
&&\text{for} \ i=0,1,\dots,n-1,&
\\
&\gamma_i=\frac{(n-i+1)(1+\omega)}{2}\qquad
&&\text{for} \ i=1,2,\dots,n,&
\\
&\theta_i=\frac{n}{2}-i
\qquad&&\text{for} \ i=0,1,\dots,n.&
\end{alignat*}
\end{Lemma}

\begin{proof}
Let $L_n'$ denote the irreducible $\K$-module obtained by twisting the $\K$-module $L_n$ via the automorphism of $\K$ given in Lemma~\ref{lem:auto_K}.
Recall the basis $\{v_i\}_{i=0}^n$ for $L_n$ from Lemma~\ref{lem:Vn_K}. Observe that the three matrices described in Lemma~\ref{lem2:Vn_K} are the matrices representing $A$, $B$, $C$ with respect to the basis $\{v_i\}_{i=0}^n$ for the $\K$-module $L_n'$.
By Theorem~\ref{thm:classification}(iii), the $\K$-module $L_n'$ is isomorphic to $L_n$. The lemma follows.
\end{proof}

Leonard pairs were introduced in \cite{LP-dual,Askeyscheme,lp&awrelation} by P.~Terwilliger. Suppose that $\omega^2\not=1$. By~Lemmas \ref{lem:Vn_K} and \ref{lem2:Vn_K}, the elements $A$ and $B$ act on the $\K$-module $L_n$ as a Leonard pair. The result was first stated in \cite[Theorem 6.3]{K&sl2}.

\subsection{The Krawtchouk algebra as a Hopf algebra}%\label{Hopf}

Let $\H$ denote an algebra.
Recall that $\H$ is called a {\it Hopf algebra} if there are two algebra homomorphisms $\e\colon \H\to \C$, $\Delta\colon \H\to \H \otimes \H$
and a linear map $S\colon \H\to \H$
that satisfy the following properties:
\begin{enumerate}\itemsep=0pt
\item[(H1)] $(1\otimes \Delta)\circ \Delta=(\Delta\otimes 1)\circ \Delta$,

\item[(H2)]
$m\circ(1\otimes (\iota \circ \e))\circ \Delta=m\circ( (\iota \circ \e)\otimes 1)\circ \Delta=1$,

\item[(H3)] $m\circ (1\otimes S)\circ \Delta=m\circ (S\otimes 1)\circ \Delta=\iota\circ \e$.
\end{enumerate}
Here $m\colon \H\otimes\H\to \H$ is the multiplication map and $\iota\colon \C\to \H$ is the unit map defined by $\iota(c)=c 1$ for all $c\in\C$.
Note that $m$ is a linear map and $\iota$ is an algebra homomorphism.

Suppose that (H1)--(H3) hold. Then the maps $\e$, $\Delta$, $S$ are called the {\it counit}, {\it comultiplication} and {\it antipode} of $\H$, respectively. Let $n$ be a positive integer.
The $n$-fold comultiplication of $\H$ is the algebra homomorphism $\Delta_n\colon \H\to \H^{\otimes (n+1)}$ inductively defined by
\begin{gather*}
\Delta_n=\big(1^{\otimes(n-1)}\otimes \Delta\big)\circ \Delta_{n-1}.
\end{gather*}
Here $\Delta_0$ is interpreted as the identity map of $\H$.
We may regard every $\H^{\otimes (n+1)}$-module as an $\H$-module by pulling back via $\Delta_n$. Note that
\begin{gather}\label{Delta_n}
\Delta_n=\big(1^{\otimes(n-i)}\otimes \Delta\otimes 1^{\otimes (i-1)}\big)\circ \Delta_{n-1}\qquad
\text{for all} \ i=1,2,\dots,n.
\end{gather}
It follows from \eqref{Delta_n} that
\begin{gather}\label{Delta_n'}
\Delta_n=(\Delta_{n-1}\otimes 1)\circ \Delta=(1\otimes \Delta_{n-1})\circ \Delta.
\end{gather}

Recall from Section~\ref{s:introduction} that $\K$ is the universal enveloping algebra of $\mathcal K_\omega$. Hence $\K$ is a Hopf algebra.
For the reader's convenience, we give a detailed verification for the Hopf algebra structure of $\K$. By an algebra antihomomorphism, we mean a unital algebra antihomomorphism.

{\samepage\begin{Lemma}\label{lem:Hopf_K}\quad
\begin{enumerate}\itemsep=0pt
\item[$(i)$] There exists a unique algebra homomorphism $\e\colon \K\to \C$ given by
\begin{gather*}
\e(A)=0,\qquad
\e(B)=0,\qquad
\e(C)=0.
\end{gather*}

\item[$(ii)$] There exists a unique algebra homomorphism $\Delta\colon \K\to \K\otimes \K$ given by
\begin{gather*}
\Delta(A)=A\otimes 1+1\otimes A,\qquad
\Delta(B)=B\otimes 1+1\otimes B,\qquad
\Delta(C)=C\otimes 1+1\otimes C.
\end{gather*}

\item[$(iii)$] There exists a unique algebra antihomomorphism $S\colon \K\to \K$ given by
\begin{gather*}
S(A)=-A,\qquad
S(B)=-B,\qquad
S(C)=-C.
\end{gather*}

\item[$(iv)$] The algebra~$\K$ is a Hopf algebra on which the counit, comultiplication and antipode are the above maps $\e$, $\Delta$, $S$, respectively.
\end{enumerate}
\end{Lemma}}

\begin{proof}
(i)--(iii) It is routine to verify Lemma~\ref{lem:Hopf_K}(i)--(iii) by using Definition \ref{defn:K}.

(iv) Using Lemma~\ref{lem:Hopf_K}(ii), it yields that $(1\otimes \Delta)\circ \Delta$ and $(\Delta\otimes 1)\circ \Delta$ agree at the generators $A$, $B$, $C$ of $\K$. Since $\Delta$ is an algebra homomorphism, the maps $(1\otimes \Delta)\circ \Delta$ and $(\Delta\otimes 1)\circ \Delta$ are algebra homomorphisms. Hence {\bf (H1)} holds for $\K$.

Let $k=m\circ (1\otimes (\iota\circ \e))\circ \Delta$ and $k'=m\circ ((\iota\circ \e)\otimes 1)\circ \Delta$. Evidently, $k$ and $k'$ are linear maps.
Using Lemma~\ref{lem:Hopf_K}(i), (ii) yields that
\begin{gather*}
k(1)=k'(1)=1,
\qquad
k(A)=k'(A)=A,
\qquad
k(B)=k'(B)=B,
\qquad
k(C)=k'(C)=C.
\end{gather*}
Let $x$, $y$ be any two elements of $\K$.
To see that $k=1$ it remains to check that $k(xy)=k(x)k(y)$. We can write
\begin{align}
&\Delta(x)=\sum_{i=1}^n x_i^{(1)}\otimes x_i^{(2)},
\label{Dx}
\\
&\Delta(y)=\sum_{i=1}^n y_i^{(1)}\otimes y_i^{(2)},
\label{Dy}
\end{align}
where $n\geq 1$ is an integer and $x_i^{(1)},x_i^{(2)},y_i^{(1)},y_i^{(2)}\in \K$ for $1\leq i\leq n$. Then
\begin{gather*}
k(xy)=
\sum_{i=1}^n \sum_{j=1}^n x_i^{(1)}\cdot y_j^{(1)}\cdot (\iota\circ \e)\big(x_i^{(2)}\big)\cdot (\iota\circ \e)\big(y_j^{(2)}\big).
\end{gather*}
Since each of $(\iota\circ \e)\big(x_i^{(2)}\big)$ and $(\iota\circ \e)\big(y_j^{(2)}\big)$ is a scalar multiple of $1$, it follows that
\begin{gather*}
k(xy)=\bigg(\sum_{i=1}^n x_i^{(1)}\cdot (\iota\circ \e)\big(x_i^{(2)}\big)\bigg)
\bigg(\sum_{j=1}^n y_j^{(1)}\cdot (\iota\circ \e)\big(y_j^{(2)}\big)\bigg)=k(x)k(y).
\end{gather*}
By a similar argument, one may show that $k'=1$. Hence (H2) holds for $\K$.

Let $h=m\circ (1\otimes S)\circ \Delta$ and $h'=m\circ (S \otimes 1)\circ \Delta$. Evidently, $h$ and $h'$ are linear maps.
Using Lemma~\ref{lem:Hopf_K}(ii), (iii) yields that
\begin{alignat*}{3}
&h(1)=h'(1)=(\iota\circ \e)(1)=1,
\qquad&&
h(A)=h'(A)=(\iota\circ \e)(A)=0,&
\\
&h(B)=h'(B)=(\iota\circ \e)(B)=0,
\qquad&&
h(C)=h'(C)=(\iota\circ \e)(C)=0.&
\end{alignat*}
Let $x$, $y$ be any two elements of $\K$ and suppose that $h(x)=(\iota\circ \e)(x)$ and $h(y)=(\iota\circ \e)(y)$.
To see that $h=\iota\circ \e$, it suffices to check that $h(xy)=h(x)h(y)$. Applying \eqref{Dx} and \eqref{Dy}, one finds that
\begin{align*}
h(xy)=\sum_{i=1}^n \sum_{j=1}^n x_i^{(1)} y_j^{(1)} S\big(x_i^{(2)} y_j^{(2)}\big).
\end{align*}
Using the antihomomorphism property of $S$, we obtain
\begin{align*}
h(xy)&=
\sum_{i=1}^n \sum_{j=1}^n x_i^{(1)} y_j^{(1)} S\big(y_j^{(2)}\big) S\big(x_i^{(2)}\big)
=\sum_{i=1}^n x_i^{(1)} \bigg(\sum_{j=1}^n y_j^{(1)} S\big(y_j^{(2)}\big)\bigg) S\big(x_i^{(2)}\big)
\\
&=\sum_{i=1}^n x_i^{(1)} h(y)S\big(x_i^{(2)}\big).
\end{align*}
Since $h(y)=(\iota\circ \e)(y)$ is a scalar multiple of $1$, it follows that
\begin{align*}
h(xy)= \sum_{i=1}^n x_i^{(1)}S\big(x_i^{(2)}\big) h(y)=h(x)h(y).
\end{align*}
By a similar argument, one can show that $h'=\iota\circ \e$. Hence (H3) holds for $\K$. The result follows.
\end{proof}

\begin{Theorem}\label{thm:CGrule_K}
For any integers $m,n\geq 0$, the $\K$-module $L_m\otimes L_n$ is isomorphic to
\[
\bigoplus_{p=0}^{\min\{m,n\}}
L_{m+n-2p}.
\]
\end{Theorem}

\begin{proof}
By Lemmas \ref{lem:Hopf_U} and \ref{lem:Hopf_K} along with Theorem~\ref{thm:KtoU} the following diagram commutes:\vspace{-1mm}
\[
\begin{tikzpicture}
\matrix(m)[matrix of math nodes,
row sep=4em, column sep=4em,
text height=1.5ex, text depth=0.25ex]
{
\K
&\U\\
\K\otimes \K
&\U\otimes \U\\
};
\path[->,font=\scriptsize,>=angle 90]
(m-1-1) edge node[left] {$\Delta$} (m-2-1)
(m-1-1) edge node[above] {$\zeta$} (m-1-2)
(m-2-1) edge node[below] {$\zeta\otimes \zeta$} (m-2-2)
(m-1-2) edge node[right] {$\Delta$} (m-2-2);
\end{tikzpicture}
\]
 Here
$\Delta\colon \U\to \U\otimes \U$ is the comultiplication of $\U$ from Lemma~\ref{lem:Hopf_U} and
$\Delta\colon \K\to \K\otimes \K$ is the comultiplication of $\K$ from Lemma~\ref{lem:Hopf_K}(ii). Combined with Theorem~\ref{thm:CGrule_U}, the result follows.
\end{proof}

For the rest of this paper, the notation $\Delta$ will refer to the map from Lemma~\ref{lem:Hopf_K}(ii) and $\Delta_n$ will stand for the corresponding $n$-fold comultiplication of $\K$ for every positive integer $n$.

\section[The Clebsch--Gordan rule for U(sl\_2) and the Hamming graph H(D,q)]{The Clebsch--Gordan rule for $\boldsymbol{\U}$\\ and the Hamming graph $\boldsymbol{H(D,q)}$}\label{s:Hamming}

\subsection{Preliminaries on distance-regular graphs}%\label{s:DRG}

Let $\Gamma$ denote a finite simple connected graph with vertex set $X\not=\varnothing$. Let $\partial$ denote the path-length distance function for $\Gamma$. Recall that the {\it diameter} $D$ of $\Gamma$ is defined by
\[
D=\max_{x,y\in X}\partial(x,y).
\]
Given any $x\in X$ let
\[
\Gamma_i(x)=\{y\in X\mid \partial(x,y)=i\}\qquad
\text{for}\ i=0,1,\dots,D.
\]
For short, we abbreviate $\Gamma(x)=\Gamma_1(x)$. We call $\Gamma$ {\it distance-regular} whenever for all $h,i,j\in\{0,1,\dots,D\}$ and all $x,y\in X$ with $\partial(x,y)=h$ the number
$|\Gamma_i(x)\cap \Gamma_j(y)|$
is independent of $x$ and $y$. If $\Gamma$ is distance-regular, the numbers $a_i$, $b_i$, $c_i$ for all $i=0,1,\dots,D$ defined by
\[
a_i = |\Gamma_i(x) \cap \Gamma(y)|,\qquad
b_i = |\Gamma_{i+1}(x) \cap \Gamma(y)|,\qquad
c_i = |\Gamma_{i-1}(x) \cap \Gamma(y)|
\]
for any $x,y\in X$ with $\partial(x,y)=i$ are called the {\it intersection numbers} of $\Gamma$. Here $\Gamma_{-1}(x)$ and $\Gamma_{D+1}(x)$ are interpreted as the empty set.

We now assume that $\Gamma$ is distance-regular.
Let $\Mat_X(\C)$ be the algebra consisting of the complex square matrices indexed by $X$.
For all $i=0,1,\dots,D$ the {\it $i^{\rm th}$ distance matrix} $\A_i\in \Mat_X(\C)$ is defined by
\[
(\A_i)_{xy}
=
\begin{cases}
1
&\text{if}\ \partial(x,y)=i,
\\
0
&\text{if}\ \partial(x,y)\not=i
\end{cases}
\]
for all $x,y\in X$. The {\it Bose--Mesner algebra} $\M$ of $\Gamma$ is the subalgebra of $\Mat_X(\C)$ generated by~$\A_i$ for all $i=0,1,\dots,D$. Note that the adjacency matrix $\A=\A_1$ of $\Gamma$ generates $\M$ and the matrices $\{\A_i\}_{i=0}^D$ form a basis for $\M$.

Since $\A$ is real symmetric and $\dim \M=D+1$, it follows that $\A$ has $D+1$ mutually distinct real eigenvalues $\theta_0,\theta_1,\dots,\theta_D$. Set $\theta_0=b_0$ which is the valency of $\Gamma$.
There exist unique $\E_0,\E_1,\dots,\E_D\in \M$ such that
\begin{gather*}
\sum_{i=0}^D \E_i=\I\qquad \text{(the identity matrix)},
\qquad
\A \E_i=\theta_i \E_i\qquad
\text{for all}\ i=0,1,\dots,D.
\end{gather*}
The matrices $\{\E_i\}_{i=0}^D$ form another basis for $\M$, and $\E_i$ is called the {\it primitive idempotent} of $\Gamma$ associated with $\theta_i$ for $i=0,1,\dots,D$.

Observe that $\M$ is closed under the Hadamard product $\odot$. The distance-regular graph $\Gamma$ is said to be {\it $Q$-polynomial} with respect to the ordering $\{\E_i\}_{i=0}^D$ if there are scalars $a_i^*$, $b_i^*$, $c_i^*$ for all $i=0,1,\dots,D$ such that $b_D^*=c_0^*=0$, $b_{i-1}^*c_i^*\not=0$ for all $i=1,2,\dots,D$ and{\samepage
\[
\E_1\odot \E_i=\frac{1}{|X|}
(b_{i-1}^* \E_{i-1}+a_i^* \E_i+ c_{i+1}^* \E_{i+1})\qquad
\text{for all}\ i=0,1,\dots,D,
\]
where we interpret $b_{-1}^*, c_{D+1}^*$ as any scalars in $\C$ and $\E_{-1},\E_{D+1}$ as the zero matrix in $\Mat_X(\C)$.}

We now assume that $\Gamma$ is $Q$-polynomial with respect to $\{\E_i\}_{i=0}^D$ and fix $x\in X$. For all $i=0,1,\dots,D$ let $\E_i^*=\E_i^*(x)$ denote the diagonal matrix in $\Mat_X(\C)$ defined by
\begin{gather}\label{e:Ei*}
(\E_i^*)_{yy}=
\begin{cases}
1 &\text{if}\ \partial(x,y)=i,
\\
0 &\text{if}\ \partial(x,y)\not=i
\end{cases}
\end{gather}
for all $y\in X$. The matrix $\E_i^*$ is called the {\it $i^{\rm th}$ dual primitive idempotent} of $\Gamma$ with respect to~$x$. The {\it dual Bose--Mesner algebra} $\M^*=\M^*(x)$ of $\Gamma$ with respect to $x$ is the subalgebra of $\Mat_X(\C)$ generated by $\E_i^*$ for all $i=0,1,\dots,D$. Since $\E_i^*\E_j^*=\delta_{ij} \E_i^*$ the matrices $\{\E_i^*\}_{i=0}^D$ form a basis for $\M^*$. For all $i=0,1,\dots,D$ the {\it $i^{\rm th}$ dual distance matrix} $\A_i^*=\A_i^*(x)$ is the diagonal matrix in $\Mat_X(\C)$ defined by
\begin{gather}\label{e:Ai*}
(\A_i^*)_{yy}=|X| (\E_i)_{xy}\qquad
\text{for all}\ y\in X.
\end{gather}
The matrices $\{\A_i^*\}_{i=0}^D$ form another basis for $\M^*$. Note that $\A^*=\A_1^*$ is called the {\it dual adjacency matrix} of $\Gamma$ with respect to $x$ and $\A^*$ generates $\M^*$ \cite[Lemma~3.11]{TerAlgebraI}.

The {\it Terwilliger algebra} $\T$ of $\Gamma$ with respect to $x$ is the subalgebra of $\Mat_X(\C)$ generated by $\M$ and $\M^*$ \cite[Definition 3.3]{TerAlgebraI}.
The vector space consisting of all complex column vectors indexed by $X$ is a natural $\T$-module and it is called the {\it standard $\T$-module} \cite[p.~368]{TerAlgebraI}.
Since the algebra $\T$ is finite-dimensional, the irreducible $\T$-modules are finite-dimensional. Since the algebra $\T$ is closed under the conjugate-transpose map, it follows that $\T$ is semisimple.
Hence the algebra $\T$ is isomorphic to
\[
\bigoplus_{\text{irreducible $\T$-modules $W$}} \End(W),
\]
where the direct sum is over all non-isomorphic irreducible $\T$-modules $W$.
Since the standard $\T$-module is faithful, all irreducible $\T$-modules are contained in the standard $\T$-module up to isomorphism.

Let $W$ denote an irreducible $\T$-module.
The number $\min_{0\leq i\leq D}\{i\mid \E_i^* W\not=\{0\}\}$ is called the {\it endpoint} of $W$.
The number
$\min_{0\leq i\leq D}\{i\mid \E_i W\not=\{0\}\}$ is called the {\it dual endpoint} of $W$.
The {\it support} of $W$ is defined as the set $\{i\mid 0\leq i\leq D,\, \E_i^* W\not=\{0\}\}$. The {\it dual support} of $W$ is defined as the set $\{i\mid 0\leq i\leq D,\, \E_i W\not=\{0\}\}$.
The number
$|\{i\mid 0\leq i\leq D,\, \E_i^* W\not=\{0\}\}|-1$
is called the {\it diameter} of $W$. The number
$|\{i\mid 0\leq i\leq D,\, \E_i W\not=\{0\}\}|-1$
is called the {\it dual diameter} of $W$.

\subsection[The adjacency matrix and the dual adjacency matrix of a Hamming grap]{The adjacency matrix and the dual adjacency matrix \\of a Hamming graph}
\label{s:A&A*_Hamming}

Let $X$ be a nonempty set and let $n$ be a positive integer.
The notation
\[
X^n=\{(x_1,x_2,\dots,x_n)\mid x_1,x_2,\dots,x_n\in X\}
\]
stands for the $n$-ary Cartesian product of $X$.
For any $x\in X^n$, let $x_i$ denote the $i^{\rm th}$ coordinate of $x$ for all $i=1,2,\dots,n$.

Recall that $q$ stands for an integer greater than or equal to $3$. For the rest of this paper, we set
\[
X=\{0,1,\dots,q-1\}
\]
and let $D$ be a positive integer.

\begin{Definition}%\label{defn:H(D)}
The {\it $D$-dimensional Hamming graph $H(D)=H(D,q)$ over $X$} has the vertex set $X^D$ and $x,y\in X^D$ are adjacent if and only if $x$ and $y$ differ in exactly one coordinate.
\end{Definition}

Let $\partial$ denote the path-length distance function for $H(D)$. Observe that $\partial(x,y)=|\{i\mid 1\leq i\leq D,\, x_i\not=y_i\}|$ for any $x,y\in X^D$.
It is routine to verify that $H(D)$ is a distance-regular graph with diameter $D$ and its intersection numbers are
\begin{gather*}
a_i=i(q-2),\qquad
b_i=(D-i)(q-1),\qquad
c_i=i
\end{gather*}
for all $i=0,1,\dots,D$.

Let $V(D)$ denote the vector space consisting of the complex column vectors indexed by $X^D$. For convenience we write $V = V (1)$. For any $x\in X^D$, let $\hat{x}$ denote the vector of $V(D)$ with $1$ in the $x$-coordinate and $0$ elsewhere. We view any $L\in \Mat_{X^D}(\C)$ as the linear map $V(D)\to V(D)$ that sends $\hat{x}$ to $L\hat{x}$ for all $x\in X^D$. We identify the vector space $V (D)$ with $V ^{\otimes D}$ via the linear isomorphism $V (D) \to V^{\otimes D}$ given by
\[
\hat{x}\to \hat{x}_1 \otimes \hat{x}_2\otimes \cdots \otimes \hat{x}_D\qquad
\text{for all}\ x\in X^D.
\]

Let $\I(D)$ denote the identity matrix in $\Mat_{X^D}(\C)$ and let $\A(D)$ denote the adjacency matrix of $H(D)$. We simply write $\I=\I(1)$ and $\A=\A(1)$.

\begin{Lemma}\label{lem:A(D)}
Let $D\geq 2$ be an integer. Then
\begin{gather}\label{rec:A(D)}
\A(D)=\A(D-1)\otimes \I+\I(D-1)\otimes \A.
\end{gather}
\end{Lemma}
\begin{proof}
Let $x\in X^D$ be given. Applying $\hat x$ to the right-hand side of \eqref{rec:A(D)} a straightforward calculation yields that it is equal to
\begin{align*}
\sum_{i=1}^{D}
\sum_{y_i\in X\setminus\{x_i\}}
\hat x_1\otimes \cdots \otimes \hat x_{i-1}
\otimes \hat y_i\otimes \hat x_{i+1}
\otimes \cdots \otimes \hat x_D=\A(D)\hat x.
\end{align*}
The lemma follows.
\end{proof}

Using Lemma~\ref{lem:A(D)}, a routine induction yields that $\A(D)$ has the eigenvalues
\[
\theta_i(D)=D(q-1)-qi\qquad
\text{for all} \ i=0,1,\dots,D.
\]
Let $\E_i(D)$ denote the primitive idempotent of $H(D)$ associated with $\theta_i(D)$ for all $i=0,1,\dots,D$. We simply write $\E_0=\E_0(1)$ and $\E_1=\E_1(1)$. For convenience, we interpret $\E_{-1}(D)$ and $\E_{D+1}(D)$ as the zero matrix in $\Mat_{X^D}(\C)$.

\begin{Lemma}\label{lem:Ei(D)}
Let $D\geq 2$ be an integer. Then
\begin{gather}\label{Ei(D)}
\E_i(D)=\E_i(D-1)\otimes \E_0+\E_{i-1}(D-1)\otimes \E_1
\qquad
\text{for all}\ i=0,1,\dots,D.
\end{gather}
\end{Lemma}

\begin{proof}
We proceed by induction on $D$.
Let $\E_i(D)'$ denote the right-hand side of \eqref{Ei(D)} for $i=0,1,\dots,D$.
Applying Lemma~\ref{lem:A(D)} along with the induction hypothesis, it follows that
\begin{gather*}
\sum_{i=0}^D \E_i(D)' =\I(D),
\qquad
\A(D) \E_i(D)'=\theta_i(D)\E_i(D)'
\qquad \text{for all}\ i=0,1,\dots,D.
\end{gather*}
Hence $\E_i(D)=\E_i(D)'$ for all $i=0,1,\dots,D$. The lemma follows.
\end{proof}

Applying Lemma~\ref{lem:Ei(D)} yields that
\[
\E_1(D)\odot \E_i(D)=
q^{-D}(b_{i-1}^* \E_{i-1}(D)+a_i^* \E_i(D)+c_{i+1}^* \E_{i+1}(D))
\]
for all $i=0,1,\dots,D$, where
\[
a_i^*=i(q-2),\qquad
b_i^*=(D-i)(q-1),\qquad
c_i^*=i
\]
for all $i=0,1,\dots,D$. Here $b_{-1}^*$, $c_{D+1}^*$ are interpreted as any scalars in $\C$.
Hence $H(D)$ is $Q$-polynomial with respect to the ordering $\{\E_i(D)\}_{i=0}^D$.

Observe that the graph $H(D)$ is vertex-transitive. Without loss of generality, we can consider the dual adjacency matrix $\A^*(D)$ of $H(D)$ with respect to $(0,0,\dots,0)\in X^D$. We simply write $\A^*=\A^*(1)$.

\begin{Lemma}\label{lem:A*(D)}
Let $D\geq 2$ be an integer. Then
\begin{gather*}
\A^*(D)=\A^*(D-1)\otimes \I+\I(D-1)\otimes \A^*.
\end{gather*}
\end{Lemma}

\begin{proof}
Given $y\in X^D$ let $c_y$ denote the coefficient of $\hat y$ in $\E_1(D) \cdot \hat 0^{\otimes D}$ with respect to the basis $\{\hat x\}_{x\in X^D}$ for $V(D)$. By \eqref{e:Ai*}, we have
\[
\A^*(D)\hat y=q^D c_y \hat y\qquad
\text{for all}\ y\in X^D.
\]

Suppose that $D\geq 2$.
Using Lemma~\ref{lem:Ei(D)} yields that
$c_y=q^{-1} c_{(y_1,\dots,y_{D-1})}+q^{1-D} c_{y_D}$ for all $y\in X^D$.
Hence
\begin{align*}
\A^*(D)\hat y
&=\big(q^{D-1} c_{(y_1,\dots,y_{D-1})}+q c_{y_D}\big) \hat y
\\
&=\A^*(D-1) (\hat y_1\otimes \cdots \otimes \hat y_{D-1})\otimes \hat y_D+\hat y_1\otimes \cdots \otimes \hat y_{D-1}\otimes \A^* \hat y_D
\\
&=(\A^*(D-1)\otimes \I+\I(D-1)\otimes \A^*) \hat y
\end{align*}
for all $y\in X^D$. The lemma follows.
\end{proof}

Let $\E_i^*(D)$ denote the $i^{{\rm th}}$ dual primitive idempotent of $H(D)$ with respect to ${(0,0,\dots,0)\!\in\! X^D}$ for all $i=0,1,\dots,D$. We simply write $\E^*_0=\E^*_0(1)$ and $\E^*_1=\E^*_1(1)$. For convenience, we interpret $\E_{-1}^*(D)$ and $\E_{D+1}^*(D)$ as the zero matrix in $\Mat_{X^D}(\C)$.

\begin{Lemma}\label{lem:Ei*(D)}
Let $D\geq 2$ be an integer. Then
\begin{gather*}
\E_i^*(D)=\E_i^*(D-1)\otimes \E^*_0+\E_{i-1}^*(D-1)\otimes \E_1^*
\qquad
\text{for all}\ i=0,1,\dots,D.
\end{gather*}
\end{Lemma}
\begin{proof}
It is straightforward to verify the lemma by using \eqref{e:Ei*}.
\end{proof}

Using Lemmas \ref{lem:A*(D)} and \ref{lem:Ei*(D)}, a routine induction yields that $\A^*(D) \E_i^*(D)=\theta_i^*(D)\E_i^*(D)$ for all $i=0,1,\dots,D$ where
$\theta_i^*(D)=D(q-1)-qi$.

\subsection{Proofs of Proposition~\ref{prop:Kpk} and Theorems~\ref{thm:dec_T(D)module}, \ref{thm:T(D)irrmodule}}

In this subsection, we set
\[
\omega=1-\frac{2}{q}.
\]
Let $\T(D)$ denote the Terwilliger algebra of $H(D)$ with respect to $(0,0,\dots,0)\in X^D$.

\begin{Definition}\label{defn:V0&V1}
Let $V_0$ denote the subspace of $V$ consisting of all vectors $\sum_{i=1}^{q-1} c_i \hat i$, where $c_1,c_2,\dots,c_{q-1}\in \C$ with $\sum_{i=1}^{q-1} c_i=0$.
Let $V_1$ denote the subspace of $V$ spanned by
$\hat 0$ and $\sum_{i=1}^{q-1} \hat i$.
\end{Definition}

\begin{Definition}\label{defn:Vs(D)}
For any $s\in \{0,1\}^D$, we define the subspace $V_s(D)$ of $V(D)$ by
\[
V_s(D)=V_{s_1}\otimes V_{s_2}\otimes \cdots \otimes V_{s_D}.
\]
Note that $V_0(1)=V_0$ and $V_1(1)=V_1$.
\end{Definition}

\begin{Lemma}\label{lem:decV(D)_HD}
The vector space $V(D)$ is equal to
\[
\bigoplus_{s\in \{0,1\}^D} V_s(D).
\]
\end{Lemma}

\begin{proof}
By Definition \ref{defn:V0&V1}, we have $V=V_0\oplus V_1$. It follows that
\[
V(D)=V^{\otimes D}=(V_0\oplus V_1)^{\otimes D}.
\]
The lemma follows by applying the distributive law of $\otimes$ over $\oplus$ to the right-hand side of the above equation.
\end{proof}

\begin{Lemma}\label{lem:r0&r1}\quad
\begin{enumerate}\itemsep=0pt
\item[$(i)$] There exists a unique representation $r_0\colon \K\to \End(V_0)$ that sends
\begin{gather*}
A \mapsto \frac{1}{q}\A|_{V_0}+\frac{1}{q},\qquad
B \mapsto \frac{1}{q}\A^*|_{V_0}+\frac{1}{q}.
\end{gather*}
Moreover, the $\K$-module $V_0$ is isomorphic to $(q-2)\cdot L_0$.

\item[$(ii)$] There exists a unique representation $r_1\colon \K\to \End(V_1)$ that sends
\begin{gather*}
A \mapsto \frac{1}{q}\A|_{V_1}+\frac{1}{q}-\frac{1}{2},\qquad
B \mapsto \frac{1}{q}\A^*|_{V_1}+\frac{1}{q}-\frac{1}{2}.
\end{gather*}
Moreover, the $\K$-module $V_1$ is isomorphic to $L_1$.
\end{enumerate}
\end{Lemma}

\begin{proof}
(i) The subspace $V_0$ of $V$ is invariant under $\A$ and $\A^*$ acting as scalar multiplication by $-1$. By Lemma~\ref{lem:Vn_K}, the statement (i) follows.

(ii) The subspace $V_1$ of $V$ is invariant under $\A$ and $\A^*$ and the matrices representing~$\A$ and~$\A^*$ with respect to the basis $\hat 0$, $\sum_{i=1}^{q-1}\hat i$ for $V_1$ are
\[
\begin{pmatrix}
0 &q-1
\\
1 &q-2
\end{pmatrix}\!,\qquad
\begin{pmatrix}
q-1 &\hphantom{-}0
\\
0 &-1
\end{pmatrix}\!,
\]
respectively. By Lemma~\ref{lem:Vn_K}, the statement (ii) follows.
\end{proof}

\begin{Definition}\label{defn:rs(D)}
For any $s\in \{0,1\}^D$, we define the representation $r_s(D)\colon \K\to \End(V_s(D))$ by
\[
r_s(D)=(r_{s_1}\otimes r_{s_2}\otimes \cdots \otimes r_{s_D})\circ \Delta_{D-1}.
\]
Note that $r_0(1)=r_0$ and $r_1(1)=r_1$.
\end{Definition}

\begin{Proposition}\label{prop:rec_rs(D)}
For any integer $D\geq 2$ and any $s\in \{0,1\}^D$, the following diagram commutes:
\[
\begin{tikzpicture}
\matrix(m)[matrix of math nodes,
row sep=3em, column sep=3em,
text height=1.5ex, text depth=0.25ex]
{
\K
&\K\otimes \K\\
\End(V_s(D))
&\\
};
\path[->,font=\scriptsize,>=angle 90]
(m-1-1) edge node[left] {$r_s(D)$} (m-2-1)
(m-1-1) edge node[above] {$\Delta$} (m-1-2)
(m-1-2) edge[bend left] node[right] {$r_{(s_1,s_2,\dots,s_{D-1})}(D-1)\otimes r_{s_D}$} (m-2-1);
\end{tikzpicture}
\]
\end{Proposition}

\begin{proof}
By Definition~\ref{defn:rs(D)} the map $r_{(s_1,s_2,\dots,s_{D-1})}(D-1)=(r_{s_1}\otimes r_{s_2}\otimes \cdots \otimes r_{s_{D-1}})\circ \Delta_{D-2}$. Hence
\begin{align*}
r_{(s_1,s_2,\dots,s_{D-1})}(D-1)\otimes r_{s_D}
&=
\left(
(r_{s_1}\otimes r_{s_2}\otimes \cdots \otimes r_{s_{D-1}})\circ \Delta_{D-2}
\right)
\otimes r_{s_D}
\\
&=(r_{s_1}\otimes r_{s_2}\otimes \cdots \otimes r_{s_D})\circ (\Delta_{D-2}\otimes 1).
\end{align*}
By \eqref{Delta_n'}, the map $\Delta_{D-1}=(\Delta_{D-2}\otimes 1)\circ \Delta$. Combined with Definition \ref{defn:rs(D)}, the following diagram commutes:
\[
\begin{tikzpicture}
\matrix(m)[matrix of math nodes,
row sep=4em, column sep=4em,
text height=1.5ex, text depth=0.25ex]
{
\K
&\K\otimes \K
&\K^{\otimes D}
\\
&\End(V_s(D))
&
\\
};
\path[->,font=\scriptsize,>=angle 90]
(m-1-1) edge[bend right] node[left] {$r_s(D)$} (m-2-2)
(m-1-1) edge node[above] {$\Delta$} (m-1-2)
(m-1-1) edge[bend left=40] node[above] {$\Delta_{D-1}$} (m-1-3)
(m-1-2) edge node[midway] {$r_{(s_1,s_2,\dots,s_{D-1})}(D-1)\otimes r_{s_D}$} (m-2-2)
(m-1-2) edge node[above] {$\Delta_{D-2}\otimes 1$} (m-1-3)
(m-1-3) edge[bend left] node[right] {$r_{s_1}\otimes r_{s_2}\otimes \cdots \otimes r_{s_D}$} (m-2-2);
\end{tikzpicture}
\]
 The proposition follows.
\end{proof}

\begin{Proposition}\label{prop:rs(D)A&B}
For any $s\in \{0,1\}^D$, the representation $r_s(D)\colon \K\to \End(V_s(D))$ maps
\begin{align}
&A\mapsto \frac{1}{q}\A(D)|_{V_s(D)}
+\frac{D}{q}-\frac{1}{2}\sum_{i=1}^D s_i,\label{rs(D)(A)}
\\
&B\mapsto \frac{1}{q}\A^*(D)|_{V_s(D)}
+\frac{D}{q}-\frac{1}{2}\sum_{i=1}^D s_i.\label{rs(D)(B)}
\end{align}
\end{Proposition}

\begin{proof}
We proceed by induction on $D$. By Lemma~\ref{lem:r0&r1}, the statement is true when $D=1$. Suppose that $D\geq 2$.
For convenience let $s'=(s_1,s_2,\dots,s_{D-1})\in \{0,1\}^{D-1}$.
By Lemma~\ref{lem:Hopf_K} and Proposition~\ref{prop:rec_rs(D)}, the map $r_s(D)$ sends $A$ to
\begin{gather*}
r_{s'}(D-1)(A)\otimes 1+1\otimes r_{s_D}(A).
\end{gather*}
Applying the induction hypothesis the above element is equal to
\begin{gather*}
\bigg(\frac{1}{q}\A(D-1)|_{V_{s'}(D-1)}+\frac{D-1}{q}-\frac{1}{2}\sum_{i=1}^{D-1} s_i\bigg)
\otimes 1+1\otimes\bigg(\frac{1}{q}\A|_{V_{s_D}}+\frac{1}{q}-\frac{s_D}{2}\bigg)
\\ \qquad
{}=\frac{\A(D-1)|_{V_{s'}(D-1)}\otimes 1+
1\otimes \A|_{V_{s_D}}}{q}+\frac{D}{q}-\frac{1}{2}\sum_{i=1}^D s_i.
\end{gather*}
By Lemma~\ref{lem:A(D)}, the first term in the right-hand side of the above equation equals $\frac{1}{q} \A(D)|_{V_s(D)}$. Hence \eqref{rs(D)(A)} holds.
By a similar argument, \eqref{rs(D)(B)} holds.
The proposition follows.
\end{proof}

In light of Proposition~\ref{prop:rs(D)A&B}, the $\T(D)$-module $V_s(D)$ is a $\K$-module for all $s\in \{0,1\}^D$. Combined with Lemma~\ref{lem:decV(D)_HD}, the standard $\T(D)$-module $V(D)$ is a $\K$-module.

\begin{Lemma}\label{lem:K1p}
Let $p$ be a positive integer. Then the $\K$-module $L_1^{\otimes p}$ is isomorphic to
\[
\bigoplus_{k=0}^{\lfloor\frac{p}{2}\rfloor}
\frac{p-2k+1}{p-k+1}{p\choose k}\cdot L_{p-2k}.
\]
\end{Lemma}
\begin{proof}
We proceed by induction on $p$. If $p=1$, then there is nothing to prove. Suppose that $p\geq 2$. Applying the induction hypothesis yields that the $\K$-module $L_1^{\otimes p}$ is isomorphic to
\begin{align*}
\bigg(\bigoplus_{k=0}^{\lfloor \frac{p-1}{2}\rfloor}
\frac{p-2k}{p-k}{p-1\choose k}\cdot L_{p-2k-1}\bigg)\otimes L_1.
\end{align*}
Applying the distributive law of $\otimes$ over $\oplus$ the above $\K$-module is isomorphic to
\begin{align*}
\bigoplus_{k=0}^{\lfloor \frac{p-1}{2}\rfloor}
\frac{p-2k}{p-k}{p-1\choose k} \cdot(L_{p-2k-1}\otimes L_1).
\end{align*}
By Theorem~\ref{thm:CGrule_K}, the $\K$-module $L_{p-2k-1}\otimes L_1$ is isomorphic to
\[
\begin{cases}
L_{p-2k}\oplus L_{p-2k-2}
&\text{if $0\leq k\leq \big\lfloor \frac{p}{2} \big\rfloor-1$},
\\
L_1&\text{else}
\end{cases}
\]
for all $k=0,1,\dots,\big\lfloor \frac{p-1}{2}\big\rfloor$. Hence the multiplicity of $L_{p-2k}$ in $L_1^{\otimes p}$ is equal to
\begin{gather*}
\frac{p-2k}{p-k}{p-1\choose k}+\frac{p-2k+2}{p-k+1}{p-1\choose k-1}=\frac{p-2k+1}{p-k+1}{p\choose k}
\end{gather*}
for all $k=0,1,\dots,\big\lfloor\frac{p}{2}\big\rfloor$. Here ${p-1 \choose k-1}$ is interpreted as $0$ when $k=0$.
The lemma follows.
\end{proof}

\begin{Lemma}\label{lem:dec_Vs(D)}
Let $p$ be an integer with $0\leq p\leq D$.
Suppose that $s\in \{0,1\}^D$ with $p=\sum_{i=1}^D s_i$. Then the $\K$-module $V_s(D)$ is isomorphic to
\[
\bigoplus_{k=0}^{\lfloor \frac{p}{2} \rfloor}
\frac{p-2k+1}{p-k+1}{p\choose k}
(q-2)^{D-p}\cdot L_{p-2k}.
\]
\end{Lemma}

\begin{proof}
By Definition \ref{defn:Vs(D)}, the $\K$-module $V_s(D)$ is isomorphic to
$V_1^{\otimes p}\otimes V_0^{\otimes (D-p)}$.
Applying Lemma~\ref{lem:r0&r1} the above $\K$-module is isomorphic to
$(q-2)^{D-p}\cdot L_1^{\otimes p}$.
Combined with Lemma~\ref{lem:K1p}, the lemma follows.
\end{proof}

\begin{proof}[Proof of Proposition~\ref{prop:Kpk}]
Let $p$ and $k$ be two integers with $0\leq p\leq D$ and $0\leq k\leq \lfloor \frac{p}{2}\rfloor$. Pick any $s\in \{0,1\}^D$ with $p=\sum_{i=1}^D s_i$.
By Lemma~\ref{lem:dec_Vs(D)}, the $\K$-module $V_s(D)$ contains the irreducible $\K$-module $L_{p-2k}$. Let $\{v_i\}_{i=0}^{p-2k}$ and $\{w_i\}_{i=0}^{p-2k}$ denote the two bases for $L_{p-2k}$ described in Lemmas \ref{lem:Vn_K} and \ref{lem2:Vn_K} with $n=p-2k$, respectively. In light of Proposition~\ref{prop:rs(D)A&B}, we may view the $\K$-submodule $L_{p-2k}$ of $V_s(D)$ as an irreducible $\T(D)$-module and denoted by $L_{p,k}(D)$.
To see~(i) and~(ii), one may evaluate the matrices representing $\A(D)$ and $\A^*(D)$ with respect to the bases $\{v_i\}_{i=0}^{p-2k}$ and $\{w_i\}_{i=0}^{p-2k}$ for $L_{p,k}(D)$, respectively. The proposition follows.
\end{proof}

\begin{proof}[Proof of Theorem~\ref{thm:dec_T(D)module}]
Let $p$ be any integer with $0\leq p\leq D$.
By Lemma~\ref{lem:dec_Vs(D)}, for any $s\in \{0,1\}^D$ with $p=\sum_{i=1}^D s_i$ the $\T(D)$-submodule $V_s(D)$ of $V(D)$ is isomorphic to
\[
\bigoplus_{k=0}^{\lfloor \frac{p}{2}\rfloor}
\frac{p-2k+1}{p-k+1}{p\choose k}(q-2)^{D-p}\cdot L_{p,k}(D).
\]
Combined with Lemma~\ref{lem:decV(D)_HD}, the result follows.
\end{proof}

\begin{proof}[Proof of Theorem~\ref{thm:T(D)irrmodule}]
Since the standard $\T(D)$-module $V(D)$ contains all irreducible \linebreak $\T(D)$-modules up to isomorphism,
 the map $\mathcal E$ is onto.
 Suppose that there are two pairs $(p,k)$ and $(p',k')$ in $\mathbf P(D)$ such that the irreducible $\T(D)$-module $L_{p,k}(D)$ is isomorphic to $L_{p',k'}(D)$. Since they have the same dimension, it follows that
\begin{gather}\label{e1}
p-2k=p'-2k'.
\end{gather}
Since $\A^*(D)$ has the same eigenvalues in $L_{p,k}(D)$ and $L_{p',k'}(D)$,
it follows from Proposition~\ref{prop:Kpk} that $p-k=p'-k'$.
Combined with~\eqref{e1}, this yields that $(p,k)=(p',k')$. Therefore, $\mathcal E$ is one-to-one.
\end{proof}

\begin{Corollary}[{\cite[Corollary 3.7]{Hamming2006}}]\label{cor:dimT(D)}
The algebra $\T(D)$ is isomorphic to
\[
\bigoplus_{p=0}^D
\bigoplus_{k=0}^{\lfloor \frac{p}{2}\rfloor}
\Mat_{p-2k+1}(\C).
\]
Moreover, $\dim \T(D)={D+4\choose 4}$.
\end{Corollary}

\begin{proof}
By Theorem~\ref{thm:T(D)irrmodule}, the algebra $\T(D)$ is isomorphic to
$\bigoplus_{p=0}^D
\bigoplus_{k=0}^{\lfloor \frac{p}{2}\rfloor}
\End(L_{p,k}(D))$.
Hence $\dim \T(D)$ is equal to
\[
\sum_{p=0}^D
\sum_{k=0}^{\lfloor \frac{p}{2}\rfloor}
(p-2k+1)^2
=\sum_{p=0}^D {p+3\choose 3}
={D+4\choose 4}.
\]
The corollary follows.
\end{proof}

\appendix
\section{Restatements of Proposition~\ref{prop:Kpk} and Theorems~\ref{thm:dec_T(D)module}, \ref{thm:T(D)irrmodule}}\label{s:restate}

Recall the irreducible $\T(D)$-module $L_{p,k}(D)$ from Proposition~\ref{prop:Kpk}. Let $r$, $r^*$, $d$, $d^*$ denote the endpoint, dual endpoint, diameter, dual diameter of $L_{p,k}(D)$ respectively. It is known from \cite[p.~197]{TerAlgebraIII} that $\big\lceil \frac{D-d}{2}\big\rceil\leq r, r^* \leq D-d$.
From the results of Section~\ref{s:A&A*_Hamming}, we see that
\begin{align*}
r=r^*=D+k-p,\qquad
d=d^*=p-2k.
\end{align*}
In terms of the parameters $r$ and $d$, the parameters $p$ and $k$ read as
\begin{align*}
p=2D-d-2r,\qquad
k=D-d-r.
\end{align*}
Thus we can restate Proposition~\ref{prop:Kpk} and Theorems~\ref{thm:dec_T(D)module}, \ref{thm:T(D)irrmodule} as follows:

\begin{Proposition}\label{prop2:Kpk}
Let $D$ be a positive integer.
For any integers $d$ and $r$ with $0\leq d\leq D$ and $\big\lceil \frac{D-d}{2}\big\rceil\leq r\leq D-d$, there exists a $(d+1)$-dimensional irreducible $\T(D)$-module $M_{d,r}(D)$ satisfying the following conditions:
\begin{enumerate}\itemsep=0pt
\item[$(i)$] There exists a basis for $M_{d,r}(D)$ with respect to which the matrices representing $\A(D)$ and~$\A^*(D)$ are
\begin{gather*}
\begin{pmatrix}
\alpha_0 &\gamma_1 & & &{\bf 0}
\\
\beta_0 &\alpha_1 &\gamma_2
\\
&\beta_1 &\alpha_2 &\ddots
 \\
& &\ddots &\ddots &\gamma_d
 \\
{\bf 0} & & &\beta_{d-1} &\alpha_d
\end{pmatrix}\!,
\qquad
\begin{pmatrix}
\theta_0 & & & &{\bf 0}
\\
 &\theta_1
\\
 & &\theta_2
 \\
 & & &\ddots
 \\
{\bf 0} & & & &\theta_d
\end{pmatrix}\!,
\end{gather*}
respectively.

\item[$(ii)$] There exists a basis for $M_{d,r}(D)$ with respect to which the matrices representing $\A(D)$ and $\A^*(D)$ are
\begin{gather*}
\begin{pmatrix}
\theta_0 & & & &{\bf 0}
\\
 &\theta_1
\\
 & &\theta_2
 \\
 & & &\ddots
 \\
{\bf 0} & & & &\theta_d
\end{pmatrix}\!,
\qquad
\begin{pmatrix}
\alpha_0 &\gamma_1 & & &{\bf 0}
\\
\beta_0 &\alpha_1 &\gamma_2
\\
&\beta_1 &\alpha_2 &\ddots
 \\
& &\ddots &\ddots &\gamma_d
 \\
{\bf 0} & & &\beta_{d-1} &\alpha_d
\end{pmatrix}\!,
\end{gather*}
respectively.
\end{enumerate}
Here the parameters $\{\alpha_i\}_{i=0}^d$,
$\{\beta_i\}_{i=0}^{d-1}$, $\{\gamma_i\}_{i=1}^d$, $\{\theta_i\}_{i=0}^d$ are as follows:
\begin{alignat*}{3}
&\alpha_i=(D-d+i-r)(q-1)-i-r\qquad
&&\text{for}\quad i=0,1,\dots,d,&
\\
&\beta_i=i+1\qquad
&&\text{for}\quad i=0,1,\dots,d-1,&
\\
&\gamma_i=(q-1)(d-i+1)\qquad
&&\text{for}\quad i=1,2,\dots,d,&
\\
&\theta_i=D(q-1)-q(i+r)\qquad
&&\text{for}\quad i=0,1,\dots,d.&
\end{alignat*}
\end{Proposition}

\begin{Theorem}\label{thm2:dec_T(D)module}
Let $D$ be a positive integer. Then the standard $\T(D)$-module $V(D)$ is isomorphic~to
\begin{gather*}
\bigoplus_{d=0}^D
\bigoplus_{r=\lceil \frac{D-d}{2}\rceil}^{D-d}
\frac{d+1}{D-r+1}{D\choose 2D-d-2r}{2D-d-2r\choose D-d-r}
(q-2)^{d-D+2r}\cdot M_{d,r}(D).
\end{gather*}
\end{Theorem}

We illustrate Theorem~\ref{thm2:dec_T(D)module} for $D=3$ and $D=4$:
\begin{table}[h!]
\centering
\extrarowheight=3pt
\begin{tabular}{c|c|c|c|c}
\toprule
$D$ &$d$ &$r$ &The support of $M_{d,r}(D)$ &The multiplicity of $M_{d,r}(D)$ in $V(D)$
\\
\midrule
\multirow{6}{*}{$3$} &$3$ &$0$ &$\{0,1,2,3\}$ &$1$
\\
\cline{2-5}
 &$2$ &$1$ &$\{1,2,3\}$ &$3(q-2)$
\\
\cline{2-5}
 &\multirow{2}{*}{$1$} &$1$ &$\{1,2\}$ &$2$
\\
\cline{3-5}
 & &$2$ &$\{2,3\}$ &$3(q-2)^2$
\\
\cline{2-5}
 &\multirow{2}{*}{$0$} &$2$ &$\{2\}$ &$3(q-2)$
 \\
\cline{3-5}
 & &$3$ &$\{3\}$ &$(q-2)^3$
 \\
 \midrule
%\end{tabular}
%\end{table}
%
%\begin{table}[H]
%\centering
%\extrarowheight=3pt
%\begin{tabular}{c|c|c|c|c}
%\toprule
%$D$ &$d$ &$r$ &The support of $M_{d,r}(D)$ &The multiplicity of $M_{d,r}(D)$ in $V(D)$
%\\
%\midrule
\multirow{9}{*}{$4$} &$4$ &$0$ &$\{0,1,2,3,4\}$ &$1$
\\
\cline{2-5}
 &$3$ &$1$ &$\{1,2,3,4\}$ &$4(q-2)$
\\
\cline{2-5}
 &\multirow{2}{*}{$2$} &$1$ &$\{1,2,3\}$ &$3$
\\
\cline{3-5}
 & &$2$ &$\{2,3,4\}$ &$6(q-2)^2$
\\
\cline{2-5}
 &\multirow{2}{*}{$1$} &$2$ &$\{2,3\}$ &$8(q-2)$
 \\
\cline{3-5}
 & &$3$ &$\{3,4\}$ &$4(q-2)^3$
 \\
 \cline{2-5}
 &\multirow{3}{*}{$0$} &$2$ &$\{2\}$ &$2$
 \\
 \cline{3-5}
 & &$3$ &$\{3\}$ &$6(q-2)^2$
 \\
 \cline{3-5}
 & &$4$ &$\{4\}$ &$(q-2)^4$
 \\
 \bottomrule
\end{tabular}
\end{table}

\begin{Theorem}\label{thm2:T(D)irrmodule}
Let $D$ be a positive integer. Let $\mathbf P(D)$ denote the set consisting of all pairs $(d,r)$ of integers with $0\leq d\leq D$ and $\big\lceil \frac{D-d}{2}\big\rceil\leq r\leq D-d$. Let $\mathbf M(D)$ denote the set of all isomorphism classes of irreducible $\T(D)$-modules. Then there exists a bijection $\mathbf P(D)\to \mathbf M(D)$ given by
\begin{gather*}
(d,r) \mapsto \text{the isomorphism class of $M_{d,r}(D)$}
\end{gather*}
for all $(d,r)\in \mathbf P(D)$.
\end{Theorem}

By Theorem~\ref{thm2:T(D)irrmodule}, the structure of an irreducible $\T(D)$-module is determined by its endpoint and its diameter. Also we can restate Corollary \ref{cor:dimT(D)} as follows:
\begin{Corollary}%\label{cor2:dimT(D)}
The algebra $\T(D)$ is isomorphic to
\[
\bigoplus_{d=0}^D\bigg(\bigg\lfloor \frac{D-d}{2}\bigg\rfloor+1\bigg)\cdot\Mat_{d+1}(\C).
\]
Moreover, $\dim \T(D)={D+4\choose 4}$.
\end{Corollary}

\subsection*{Acknowledgements}
The author would like to thank the anonymous referees for insightful suggestions to improve the paper and bring his attention to \cite{K&sl2}. Also, the author thanks Dr.\ Luc Vinet to bring his attention to \cite{Johnson:2021, Hamming:2021,Hamming:2015}.
The research is supported by the Ministry of Science and Technology of Taiwan under the project MOST 110-2115-M-008-008-MY2.

\pdfbookmark[1]{References}{ref}
\LastPageEnding


\begin{thebibliography}{99}
\footnotesize\itemsep=0pt

\bibitem{Johnson:2021}
Bernard P.-A., Cramp\'e N., Vinet L., Entanglement of free fermions on Johnson
 graphs, \href{https://arxiv.org/abs/2104.11581}{arXiv:2104.11581}.

\bibitem{Hamming:2021}
Bernard P.-A., Cramp\'e N., Vinet L., Entanglement of free fermions on {H}amming
 graphs, \href{https://doi.org/10.1016/j.nuclphysb.2022.116061}{\textit{Nuclear Phys.~B}} \textbf{986} (2023), 116061, 22~pages,
 \href{https://arxiv.org/abs/2103.15742}{arXiv:2103.15742}.

\bibitem{Wedderburn1962}
Curtis C.W., Reiner I., Representation theory of finite groups and associative
 algebras, \textit{Pure Appl. Math.}, Vol.~11, Interscience Publishers, New
 York, 1962.

\bibitem{semidefinite2006}
Gijswijt D., Schrijver A., Tanaka H., New upper bounds for nonbinary codes
 based on the {T}erwilliger algebra and semidefinite programming,
 \href{https://doi.org/10.1016/j.jcta.2006.03.010}{\textit{J.~Combin. Theory Ser.~A}} \textbf{113} (2006), 1719--1731.

\bibitem{hypercube2002}
Go J.T., The {T}erwilliger algebra of the hypercube,
 \href{https://doi.org/10.1006/eujc.2000.0514}{\textit{European~J.~Combin.}} \textbf{23} (2002), 399--429.

\bibitem{Huang:2015}
Huang H.-W., Finite-dimensional irreducible modules of the universal
 {A}skey--{W}ilson algebra, \href{https://doi.org/10.1007/s00220-015-2467-9}{\textit{Comm. Math. Phys.}} \textbf{340} (2015),
 959--984, \href{https://arxiv.org/abs/1210.1740}{arXiv:1210.1740}.

\bibitem{Huang:BImodule}
Huang H.-W., Finite-dimensional irreducible modules of the {B}annai--{I}to
 algebra at characteristic zero, \href{https://doi.org/10.1007/s11005-020-01306-9}{\textit{Lett. Math. Phys.}} \textbf{110}
 (2020), 2519--2541, \href{https://arxiv.org/abs/1910.11447}{arXiv:1910.11447}.

\bibitem{SH:2019-1}
Huang H.-W., Bockting-Conrad S., Finite-dimensional irreducible modules of the
 {R}acah algebra at characteristic zero, \href{https://doi.org/10.3842/SIGMA.2020.018}{\textit{SIGMA}} \textbf{16} (2020),
 018, 17~pages, \href{https://arxiv.org/abs/1910.11446}{arXiv:1910.11446}.

\bibitem{Hamming:2015}
Jafarizadeh M.A., Nami S., Eghbalifam F., Entanglement entropy in the {H}amming
 networks, \href{https://arxiv.org/abs/1503.04986}{arXiv:1503.04986}.

\bibitem{kassel}
Kassel C., Quantum groups, \textit{Grad. Texts in Math.}, Vol.~155, \href{https://doi.org/10.1007/978-1-4612-0783-2}{Springer},
 New York, 1995.

\bibitem{Hamming2006}
Levstein F., Maldonado C., Penazzi D., The {T}erwilliger algebra of a {H}amming
 scheme {$H(d,q)$}, \href{https://doi.org/10.1016/j.ejc.2004.08.005}{\textit{European~J.~Combin.}} \textbf{27} (2006), 1--10.

\bibitem{Hopf65}
Milnor J.W., Moore J.C., On the structure of {H}opf algebras, \href{https://doi.org/10.2307/1970615}{\textit{Ann. of
 Math.}} \textbf{81} (1965), 211--264.

\bibitem{K&sl2}
Nomura K., Terwilliger P., Krawtchouk polynomials, the {L}ie algebra
 {$\mathfrak{sl}_2$}, and {L}eonard pairs, \href{https://doi.org/10.1016/j.laa.2012.02.006}{\textit{Linear Algebra Appl.}}
 \textbf{437} (2012), 345--375, \href{https://arxiv.org/abs/1201.1645}{arXiv:1201.1645}.

\bibitem{Doob1997}
Tanabe K., The irreducible modules of the {T}erwilliger algebras of {D}oob
 schemes, \href{https://doi.org/10.1023/A:1008647205853}{\textit{J.~Algebraic Combin.}} \textbf{6} (1997), 173--195.

\bibitem{LP-dual}
Terwilliger P., Leonard pairs and dual polynomial sequences, {U}npublished
 manuscript, 1987, available at
 \url{https://www.math.wisc.edu/~terwilli/Htmlfiles/leonardpair.pdf}.

\bibitem{TerAlgebraI}
Terwilliger P., The subconstituent algebra of an association scheme.~{I},
 \href{https://doi.org/10.1023/A:1022494701663}{\textit{J.~Algebraic Combin.}} \textbf{1} (1992), 363--388.

\bibitem{TerAlgebraII}
Terwilliger P., The subconstituent algebra of an association scheme.~{II},
 \href{https://doi.org/10.1023/A:1022480715311}{\textit{J.~Algebraic Combin.}} \textbf{2} (1993), 73--103.

\bibitem{TerAlgebraIII}
Terwilliger P., The subconstituent algebra of an association scheme.~{III},
 \href{https://doi.org/10.1023/A:1022415825656}{\textit{J.~Algebraic Combin.}} \textbf{2} (1993), 177--210.

\bibitem{Askeyscheme}
Terwilliger P., An algebraic approach to the {A}skey scheme of orthogonal
 polynomials, in Orthogonal Polynomials and Special Functions, \href{https://doi.org/10.1007/978-3-540-36716-1_6}{\textit{Lecture
 Notes in Math.}}, Vol.~1883, Springer, Berlin, 2006, 255--330.

\bibitem{Manila} Terwilliger P., Manila notes, 2010, available at \url{https://people.math.wisc.edu/~terwilli/teaching.html}.

\bibitem{lp&awrelation}
Terwilliger P., Vidunas R., Leonard pairs and the {A}skey--{W}ilson relations,
 \href{https://doi.org/10.1142/S0219498804000940}{\textit{J.~Algebra Appl.}} \textbf{3} (2004), 411--426,
 \href{https://arxiv.org/abs/math.QA/0305356}{arXiv:math.QA/0305356}.

\bibitem{Vidunas:2007}
Vid\~unas R., Normalized {L}eonard pairs and {A}skey--{W}ilson relations,
 \href{https://doi.org/10.1016/j.laa.2005.12.033}{\textit{Linear Algebra Appl.}} \textbf{422} (2007), 39--57,
 \href{https://arxiv.org/abs/math.RA/0505041}{arXiv:math.RA/0505041}.

\end{thebibliography}
\end{document}